\documentclass[10pt]{amsart}

\usepackage[colorlinks]{hyperref}
\usepackage{color,graphicx,shortvrb}
\usepackage[latin 1]{inputenc}
\usepackage[dvipsnames]{xcolor}

\usepackage[active]{srcltx} %SRC Specials for DVI search

\usepackage{enumerate}
\usepackage{amssymb}
\usepackage{tikz-cd}

\usepackage [ all ]{xy}

\newtheorem{theorem}{Theorem}[section]

\newtheorem{corollary}[theorem]{Corollary}

\newtheorem{lemma}[theorem]{Lemma}

\newtheorem{proposition}[theorem]{Proposition}
\newtheorem{remark}[theorem]{Remark}

\def\J#1#2#3{ \left\{ #1,#2,#3 \right\} }

\def\11{\textbf{$1$}}
\def\11b#1{\mathbf{1}_{_{#1}}}
\def\CC{{\mathbb{C}}}

\DeclareGraphicsExtensions{.jpg,.pdf,.png,.eps}

\begin{document}

\title[Preservers of triple transition pseudo-probabilities]{Preservers of triple transition pseudo-probabilities in connection with orthogonality preservers and surjective isometries}

\author[A.M. Peralta]{Antonio M. Peralta}

\address{Instituto de Matemáticas de la Universidad de Granada (IMAG). Departamento de An{\'a}lisis Matem{\'a}tico, Facultad de	Ciencias, Universidad de Granada, 18071 Granada, Spain.}

\email{aperalta@ugr.es}

%\thanks{}



\date{}

\begin{abstract} We prove that every bijection preserving triple transition pseudo-probabilities between the sets of minimal tripotents of two atomic JBW$^*$-triples automatically preserves orthogonality in both directions. Consequently, each bijection  preserving triple transition pseudo-probabilities between the sets of minimal tripotents of two atomic JBW$^*$-triples is precisely the restriction of a (complex-)linear triple isomorphism between the corresponding JBW$^*$-triples. This result can be regarded as triple version of the celebrated Wigner theorem for Wigner symmetries on the posets of minimal projections in $B(H)$. We also present a 
Tingley type theorem by proving that every surjective isometry between the sets of minimal tripotents in two atomic JBW$^*$-triples admits an extension to a real linear surjective isometry between these two JBW$^*$-triples. We also show that the class of surjective isometries between the sets of minimal tripotents in two atomic JBW$^*$-triples is, in general, strictly wider than the set of bijections preserving triple transition pseudo-probabilities. 
\end{abstract}

\maketitle
\thispagestyle{empty}

\section{Introduction}

As it is masterfully narrated by G. Chevalier in \cite{Chev2007}, the Geneva-Brussels School proposed the orthomodular lattice $\mathcal{P}(H)$, of all projections on a complex Hilbert space $H,$ with the partial order defined by $p\leq q$ in $\mathcal{P}(H)$ if $p q =p$ and orthogonality determined by zero product --equivalently, the orthomodular lattice $\mathbf{L}$ of all closed subspaces of $H$ with the partial ordering given by inclusion and orthogonality in the Euclidean sense-- as mathematical model in quantum mechanics.  By a result due to C. Piron, from 1976, every isomorphism of the propositional system of all closed subspaces of a complex Hilbert space of dimension at least 3 is induced by a unitary or by an antiunitary operator (see \cite[Corollary 13]{Chev2007}). This is actually an equivalent reformulation of the celebrated Wigner's unitary-antiunitary theorem (cf. \cite{CasdeVilahtiLevrero97}).\smallskip

The elements in $\mathcal{P}(H)$ are precisely the positive partial isometries in $B(H)$. We recall that an operator $e$ in $B(H)$ is called a \emph{partial isometry} or a \emph{tripotent} if $e e^* e =e$ (equivalently, $ee^*$ or $e^* e$ lies in $\mathcal{P}(H)$). Along this note we shall write $PI(H)= \mathcal{U}(B(H))$ for the collection of all partial isometries in $B(H)$. In \cite{Molnar2002} L. Moln{\'a}r paved a new ground to establish an analogue to Piron's version of Wigner's unitary-antiunitary theorem for preservers of order and orthogonality between the corresponding structures of partial isometries. We should recall first that for $e,v\in \mathcal{U}(B(H))$ we write $e\leq v$ (respectively, $e$ is orthogonal to $v$, $e\perp v$ in short) if $ee^*\leq vv^*$ and $e^* e \leq v^* v$ (respectively, $e v^* = v^* e =0$). Moln{\'a}r's version of the Piron-Wigner theorem asserts that for each complex Hilbert space $H$ with dim$(H)\geq 3,$ every bijective transformation $\Phi : \mathcal{U}(B(H))\to \mathcal{U}(B(H))$ which preserves the partial ordering and orthogonality between partial isometries in both directions and is continuous {\rm(}in the operator norm{\rm)} at a single element of $\mathcal{U}(B(H))$ different from $0$, extends to a real-linear triple isomorphism (cf. \cite[Theorem 1]{Molnar2002}).\smallskip

The Banach space $B(H)$, of all bounded linear operators on a complex Hilbert space $H$, is more than a prototype of C$^*$- and von Neumann algebras. By the Gelfand-Naimark theorem every C$^*$-algebra embeds as a self-adjoint subalgebra of some $B(H)$. It is known that $B(H)$ can be also regarded as  particular case of type 1 Cartan factors. There are six different types of Cartan factors (see section \ref{subsect: definitions} for definitions), which are employed in a Gelfand-Naimark type theorem to represent every JB$^*$-triple isometrically as a JB$^*$-subtriple of an $\ell_{\infty}$-sum of Cartan factors (cf. \cite{FriRu86}).\smallskip

As we shall see below, those complex Banach spaces whose open unit ball is a bounded symmetric domain were characterized by W. Kaup in \cite{Ka83} as the complex  Banach spaces $E$ admitting a continuous triple product $\{.,.,.\} : E\times E\times E\to E$ (bilinear and symmetric in the outer variables and conjugate linear in the middle one) satisfying a collection of algebraic and analytic axioms (see section \ref{subsect: definitions}). For the moment we shall simply note that every C$^*$-algebra is a JB$^*$-triple for the triple product\begin{equation}\label{eq triple product JCstar triple} \{x,y,z\} = \frac12\left( x y^* z + z y^* x\right),
\end{equation}, while the class of JB$^*$-triples is strictly wider than the collection of all C$^*$-algebras since it also contains, among other more exotic examples, all infinite dimensional complex Hilbert spaces. The fixed-points of the just commented triple product on a C$^*$-algebra $A$ are the partial isometries in $A$. The elements $e$ in a JB$^*$-triple $E$ satisfying $\{e,e,e\}= e$ are called tripotents. The set of all tripotents in $E$ will be denoted by $\mathcal{U} (E)$. As it will be detailed in section \ref{subsect: definitions}, there is a natural partial order and a notion of orthogonality for elements in $\mathcal{U} (E)$, which restricted to $\mathcal{U} (B(H))$ are precisely those employed by Moln{\'a}r in the theorem commented above.\smallskip

A triple homomorphism between JB$^*$-triples $E$ and $F$ is a linear map $T:E\to F$ preserving triple products. Every triple homomorphism between JB$^*$-triples is automatically continuous \cite[Lemma 1]{BarDanHor88}. If $T$ is a triple isomorphism (i.e. a bijective triple homomorphism), its restriction $T|_{\mathcal{U} (E)}: \mathcal{U} (E)\to \mathcal{U} (F)$ is a surjective isometry $T|_{\mathcal{U} (E)}: \mathcal{U} (E)\to \mathcal{U} (F)$ which preserves orthogonality and partial order in both directions (we recall that very injective triple homomorphism is an isometry \cite[Lemma 1]{BarDanHor88}). As we shall justify later, the restriction of $T$ to the corresponding subsets of all minimal tripotents is also a surjective isometry.\smallskip

As in the case of C$^*$-algebras, there exist JB$^*$-triples $E$ for which $\mathcal{U} (E) = \{0\}.$ However, in a JB$^*$-triple $E$ the extreme points of its closed unit ball are precisely the complete tripotents in $E$ (cf. \cite[Lemma 4.1]{BraKaUp78}, \cite[Proposition 3.5]{KaUp77} or \cite[Corollary 4.8]{EdRutt88}). Thus, every JB$^*$-triple which is also a dual Banach space contains an abundant set of tripotents. JB$^*$-triples which are additionally dual Banach spaces are called \emph{JBW$^*$-triples}. Each JBW$^*$-triple admits a unique (isometric) predual and its triple product is separately weak$^*$ continuous (cf. \cite{BarTi}). Since each Cartan factor is a dual Banach space, each $\ell_{\infty}$-sum of Cartan factors is a JBW$^*$-triple. The JBW$^*$-triples which are of this form are called atomic JBW$^*$-triples.\smallskip

In a recent collaboration with Y. Friedman, we studied bijective transformations preserving orthogonality and order between the sets of tripotents of two atomic JBW$^*$-triples. More concretely, let $\displaystyle M= \bigoplus_{i\in I}^{\ell_{\infty}} C_i$ and  $\displaystyle N = \bigoplus_{j\in J}^{\ell_{\infty}} \tilde{C}_j$ be atomic JBW$^*$-triples, where $C_i$ and $C_j$ are Cartan factors with rank $\geq 2$. Suppose that $\Phi : \mathcal{U}(M) \to \mathcal{U}(N)$ is a bijective transformation which preserves the partial ordering in both directions and orthogonality between tripotents. If we additionally assume that $\Phi$ is continuous at a tripotent $u = (u_i)_i$ in $M$ with $u_i\neq 0$ for all $i$ {\rm(}or we simply assume that $\Phi|_{\mathbb{T} u}$ is continuous at a tripotent $(u_i)_i$ in $M$ with $u_i\neq 0$ for all $i${\rm)}, then there exists a real linear triple isomorphism $T: M\to N$ such that $T(w) = \Phi(w)$ for all $w\in \mathcal{U} (M)$ (cf. \cite[Theorem 6.1]{FriPe2021}).\smallskip

Back to the original statement of Wigner's theorem, we recall the notion of transition probability between minimal (i.e. rank-one) projections in $B(H)$. Suppose $p= \xi\otimes \xi$ and $q= \eta\otimes \eta$ are two minimal projections in $B(H)$ with $\xi$ and $\eta$ in the unit sphere of $H$ (where $\xi\otimes \eta (\zeta) := \langle \zeta, \eta\rangle \xi$). The \emph{transition probability} from $p$ to $q$ is defined as
$$TP(p,q)=\hbox{tr}(p q) = \hbox{tr}(pq^*) = \hbox{tr}(q p^*) = |\langle \xi, \eta\rangle|^2.$$ Let $\mathcal{P}_1 (H)$ stand for the set of all minimal projections in $B(H)$. A bijective map $\Phi : \mathcal{P}_1 (H) \to \mathcal{P}_1 (H)$ is called a \emph{symmetry transformation} or a \emph{Wigner symmetry} if it preserves the transition probability between minimal projections, that is, $$TP (\Phi(p),\Phi(q)) = \hbox{tr} ( \Phi(p) \Phi(q)) = \hbox{tr} (pq) = TP(p,q),\hbox{ for all } (p,q \in \mathcal{P}_1 (H)).$$ Wigner's theorem proves that symmetry transformations on $\mathcal{P}_1(H)$ are characterized as those bijective maps $\Phi : \mathcal{P}_1(H) \to \mathcal{P}_1(H)$ for which there is an either unitary (i.e. a linear mapping $u : H\to H$ such that $u u ^* = u^* u =1$) or an antiunitary (i.e. a conjugate-linear mapping $u : H\to H$ such that $u u ^* = u^* u =1$) operator $u$ on $H$, unique up to multiplication by a unitary scalar, such that $\Phi (p) = u p u^*$ for all $p\in \mathcal{P}_1(H)$ (cf. \cite{Wig31}, \cite[page 12]{Molnar85}). \smallskip

The following version of Wigner's theorem for minimal partial isometries is also due to L. Moln{\'a}r. In the statement we see that the set of minimal projections in $B(H)$ has been enlarged to the set, $\mathcal{U}_{min} (B(H)),$ of all minimal partial isometries in $B(H)$. We recall that a partial isometry $e$ in $B(H)$ is called minimal if its left projection $e e^*$ (equivalently, its right projection $e^* e$) is minimal.

\begin{theorem}\label{t Molnar minimal pi}{\rm\cite[Theorem 2]{Molnar2002}} Let $\Phi: \mathcal{U}_{min} (B(H))\to \mathcal{U}_{min} (B(H))$ be a bijective mapping satisfying \begin{equation}\label{equation in Molnar's teorem} \hbox{\rm tr} ({\Phi(e)}^* \Phi(v)) = \hbox{\rm tr}(e^* v), \hbox{ for all } e,v\in \mathcal{U}_{min} (B(H)).
	\end{equation}
	Then $\Phi$ extends to a surjective complex-linear isometry. Moreover, one of the following statements holds: \begin{enumerate}[$(a)$]
		\item There exist unitaries $u, w$ on $H$ such that $\Phi(e) = u e w$ {\rm(}$e \in \mathcal{U}_{min}(B(H))${\rm)};
		\item There exist antiunitaries $u, w$ on $H$ such that $\Phi(e) = u e^* w$ {\rm(}$e \in \mathcal{U}_{min}(B(H))${\rm)}.
	\end{enumerate}
\end{theorem}

The transition probability between two minimal projections $p,q$ in $B(H)$ coincides with tr$(p q^*)\in [0,1],$ so the hypothesis assumed by Moln{\'a}r in \eqref{equation in Molnar's teorem} (i.e. preservation of tr$(e^* v)\in \mathbb{C}$) is an analogue of transition probability preservation for non-necessarily positive minimal partial isometries. Let us analyse this new ``generalized transition probability''. If we fix a minimal partial isometry $e$ in $B(H),$ the functional $\varphi_e (x) = \hbox{tr} (e^* x)$ is the unique extreme point of the closed unit ball of $B(H)_*$, the predual of $B(H)$, at which $e$ attains its norm, so $\hbox{tr} (e^* v) = \varphi_e (v) $. This is the crucial point to consider the notion of triple transition pseudo-probability from a minimal tripotent to another minimal tripotent in an arbitrary JBW$^*$-triple as introduced in the recent reference \cite{Pe2022TTP}.  More concretely, for each minimal tripotent $e$ in a JBW$^*$-triple, $M,$ there exists a unique pure atom (i.e. an extreme point of the closed unit ball of $M_*$) $\varphi_e$ at which $e$ attains its norm and the corresponding Peirce-2 projection writes in the form $P_2 (e) (x) = \varphi_e(x) e$ for all $x\in M$ (cf. \cite[Proposition 4]{FriRu85}). The mapping $$\mathcal{U}_{min} (M)\to \partial_{e} (\mathcal{B}_{M_*}), \ \  e\mapsto \varphi_e $$ is a bijection from the set of minimal tripotents in $M$ onto the set of pure atoms of $M$. Given two minimal tripotents $e$ and $v$ in a JBW$^*$-triple $M$, we define the \emph{triple transition pseudo-probability} from $e$ to $v$ as the complex number given by \begin{equation}\label{eq triple transition pseudo prbability} TTP(e,v)= \varphi_v(e).
\end{equation} We can no longer use the term ``probability'' because $TTP(e,v)$ is an element in the closed unit ball of the complex plane. In the case of $B(H)$, the triple transition pseudo-probability between two minimal projections is precisely the usual transition probability in Wigner's theorem, while the hypothesis \eqref{equation in Molnar's teorem} in Theorem \ref{t Molnar minimal pi} simply says that $\Phi$ preserves triple transition pseudo-probabilities.\smallskip

We recall for later purposes that the triple transition pseudo-probability is symmetric in the sense that $TTP(e,v)= \overline{TTP(v,e)},$ for every couple of minimal tripotents $e,v\in M$ (see \cite[$(2.3)$]{Pe2022TTP}). \smallskip

In view of Theorem \ref{t Molnar minimal pi}, it is an attractive challenge to ask whether a bijection $\Phi$ between the sets of minimal tripotents of two Cartan factors  (or more generally of two atomic JBW$^*$-triples) $M$ and $N$ preserving triple transition pseudo-probabilities is precisely the restriction of a (complex-)linear triple isomorphism between the corresponding JBW$^*$-triples. This problem has been positively solved when $M$ and $N$ are both Cartan factors of type 1 (i.e. Banach spaces $B(H,K)$ of bounded linear operators between complex Hilbert spaces) or when $M$ and $N$ are both type 4 or spin Cartan factors (see \cite[Theorems 4.4 and 3.2]{Pe2022TTP}). It is worth to note that the proof of the results is built upon classic theorems on preservers and concrete tools for operator spaces and Hilbert spaces. The general problem remains open.\smallskip

This paper presents a complete solution to the problem just presented (see   Corollary \ref{c bijections preserving triple transition pseudo-probabilities and orthogonality}). Here, instead of combining classical tools on preservers for concrete Cartan factors, we shall turn our point of view to a completely newfangled strategy with arguments and tools taken from abstract theory of JB$^*$-triples. As we shall see in section \ref{subsect: definitions}, the achievements in \cite[Theorem 2.3]{Pe2022TTP} prove that each bijective transformation $\Phi$ preserving triple transition pseudo-probabilities between the sets of minimal tripotents of two atomic JBW$^*$-triples $M$ and $N$, admits an extension to a bijective {\rm(}complex{\rm)} linear mapping $T_0$ from the socle of $M$ onto the socle of $N$ whose restriction to $\mathcal{U}_{min} (M)$ is $\Phi$, where the socle of a JB$^*$-triple is the subspace linearly generated by its minimal tripotents. If we additionally assume that $\Phi$ preserves orthogonality,  then $\Phi$ admits an extension to a surjective (complex-)linear {\rm(}isometric{\rm)} triple isomorphism from $M$ onto $N$ (cf. \cite[Corollary 2.5]{Pe2022TTP}). In Theorem \ref{t preservers of ttp also preserve orthog} we prove that every bijection preserving triple transition pseudo-probabilities between the sets of minimal tripotents of two atomic JBW$^*$-triples automatically preserves orthogonality in both directions. \smallskip

The main conclusion in this paper shows that the set of minimal tripotents in an atomic JBW$^*$-triple together with the triple transition pseudo-probabilities among its elements is a complete invariant valid to determine the whole structure of the JB$^*$-triple (cf. Corollary \ref{c bijections preserving triple transition pseudo-probabilities and orthogonality}). The result should be complemented with the main conclusion of \cite{FriPe2021}, which asserts that in an atomic JBW$^*$-triple $M$ containing no rank-one Cartan factors the poset of all tripotents in $M$ with the partial order and the relation of orthogonality is a complete invariant for its structure of real JBW$^*$-triple. Both results together validate the full analogy in the setting of JB$^*$-triples with the different statements of Wigner's theorem for projections.\smallskip

Another result derived from our main conclusion (see Corollary \ref{c bijections preserving triple transition pseudo-probabilities are isometries}) proves that every bijection preserving triple transition pseudo-probabilities between the sets of minimal tripotents in two atomic JBW$^*$-triples is an isometry with respect to the gap metric (i.e. the metric given by the JB$^*$-triple norm).\smallskip

It is known that the gap metric and the usual transition probability between minimal projections in $B(H)$ are mutually determined (see \cite[$(2.6.13)$ in page 127]{Molnar85} or \cite{Geher2014}). However, as we shall see in Remark \ref{r TTP and distance are not mutually determined}, the triple transition pseudo-probability and the gap metric are not, in general, related each other. It naturally arises the problem of studying those bijections preserving distances between the sets of minimal tripotents in two atomic JBW$^*$-triples. This task is culminated in Theorem \ref{t Tingley for minimal in atomic JBW-triples}. In the just quoted result we establish a variant of Tingley's theorem by proving that every surjective isometry between the sets of minimal tripotents in two atomic JBW$^*$-triples admits an extension to a real linear surjective isometry between these two JBW$^*$-triples. The proof is obtained by an application of the result describing the bijections preserving triple transition pseudo-probabilities between sets of minimal tripotents in atomic JBW$^*$-triples. However, the class of surjective isometries between the sets of minimal tripotents in two atomic JBW$^*$-triples is, in general, strictly wider than the set of bijections preserving triple transition pseudo-probabilities, since we can also find examples of extensions which are conjugate linear or which are neither complex linear nor conjugate linear.

\section{Background and state-of-the-art}\label{subsect: definitions}

This section is aimed to provide the reader with the basic terminology and notions to understand the results and to fill the gaps left in the introduction.  We shall also approach to a brief state-of-the-art of the main problem tackled in this paper.\smallskip

Our arguments will employ tools developed in theory of JB$^*$-triples. So, it seems pertinent to recall the definition of JB$^*$-triple (cf. \cite{Ka83}), a mathematical model originally arisen in holomorphic theory deeply studied in functional analysis.\smallskip

A \emph{JB$^*$-triple} is a complex Banach space $E$ together with a continuous triple product $\J \cdot\cdot\cdot :
E\times E\times E \to E,$ which is symmetric and bilinear in the first and third variables, conjugate-linear in the middle one, and satisfies the following axioms:
\begin{enumerate}[{\rm (a)}] \item (Jordan identity)
	$$L(a,b) L(x,y) = L(x,y) L(a,b) + L(L(a,b)x,y)
	- L(x,L(b,a)y)$$ for $a,b,x,y$ in $E$, where $L(a,b)$ is the operator on $E$ given by $x \mapsto \J abx;$
	\item $L(a,a)$ is a hermitian operator with non-negative spectrum for all $a\in E$;
	\item $\|\{a,a,a\}\| = \|a\|^3$ for each $a\in E$.\end{enumerate}

The examples of mathematical models included in the class of JB$^*$-triples is perhaps one of the biggest attractiveness of this notion. We have already commented that every C$^*$-algebra is a JB$^*$-triple. The same triple product employed for C$^*$-algebras given in \eqref{eq triple product JCstar triple} serves to equip the space $B(H,K),$ of all bounded linear operators between two complex Hilbert spaces $H$ and $K$, with a structure of JB$^*$-triple. The JB$^*$-triples of the form $B(H,K)$ are known as \emph{Cartan factors of type 1}. There are another 5 types of Cartan factors. Cartan factors of types 2 and 3 are subtriples of $B(H)$ defined in the following way. Fix a conjugation $j$ (i.e. a conjugate-linear isometry or period 2) on a complex Hilbert space $H$, and define a linear involution on $B(H)$ by $x\mapsto x^t:=jx^*j$ --this is an infinite dimensional version of the transposition in $M_n(\mathbb{C})$. \emph{Cartan factors of type 2 and 3} are the JB$^*$-subtriples of $B(H)$ of all $t$-skew-symmetric and $t$-symmetric operators, respectively.\smallskip

A \emph{Cartan factor of type 4}, also called a \emph{spin factor},\label{def spin factor} is a complex Hilbert space $M$ provided with a conjugation $x\mapsto \overline{x},$ where the triple product and the norm are defined by \begin{equation}\label{eq spin product}
\{x, y, z\} = \langle x, y\rangle z + \langle z, y\rangle  x -\langle x, \overline{z}\rangle \overline{y},
\end{equation} and \begin{equation}\label{eq spin norm} \|x\|^2 = \langle x, x\rangle  + \sqrt{\langle x, x\rangle ^2 -|
\langle x, \overline{x}\rangle  |^2},
 \end{equation} respectively (cf. \cite[Chapter 3]{Fri2005}). The \emph{Cartan factors of types 5 and 6} (also called \emph{exceptional} Cartan factors) are spaces of matrices over the eight dimensional complex algebra of Cayley numbers; the type 6 consists of all $3\times 3$ self-adjoint matrices and has a natural Jordan algebra structure, and the type 5 is the subtriple consisting of all $1\times 2$ matrices (see \cite{Ka97, Harris74, HervIs92} and the recent references \cite[\S 6.3 and 6.4]{HamKalPe20}, \cite[\S 3]{HamKalPe22determinants} for more details).\smallskip

As we have already commented during the introduction, partial isometries in a C$^*$-algebras $A$ are precisely the elements which are fixed points for its natural triple product \eqref{eq triple product JCstar triple}. The fixed points of the triple product of a JB$^*$-triple $E$ are called \emph{tripotents}.  We write $\mathcal{U} (E)$ for the set of all tripotents in $E$. Each $e$ in $\mathcal{U} (E)$ produces the following \emph{Peirce decomposition} of the space $E$ in terms of the eigenspaces of the operator $L(e,e)$:
\begin{equation}\label{Peirce decomp} {E} = {E}_{0} (e) \oplus  {E}_{1} (e) \oplus {E}_{2} (e),\end{equation} where ${
E}_{k} (e) := \{ x\in {E} : L(e,e)x = {\frac k 2} x\}$ is a subtriple of ${E}$ called the \emph{Peirce-$k$ subspace} ($k=0,1,2$). \emph{Peirce-$k$ projection} is the name given to the natural projection of ${E}$ onto ${E}_{k} (e)$ and it is usually denoted by $P_{k} (e)$. %We shall apply later that Peirce projections are all contractive (cf. \cite{FriRu85}).
Triple products among elements in different Peirce subspaces obey certain laws known as \emph{Peirce arithmetic}. Concretely,
the inclusion $\J {{E}_{k}(e)}{{E}_{l}(e)}{{E}_{m}(e)}\! \subseteq {E}_{k-l+m} (e),$ and the identity $\J {{E}_{0}(e)}{{E}_{2} (e)}{{E}}\! =\! \J {{E}_{2} (e)}{{E}_{0} (e)}{{E}}\! =\! \{0\},$ hold for all $k,l,m\in \{0,1,2\}$, where ${E}_{k-l+m} (e) = \{0\}$ whenever $k-l+m$ is not in $\{0,1,2\}$.\smallskip  

The Peirce-$2$ subspace ${E}_{2} (e)$ is a unital JB$^*$-algebra with respect to the product and involution given by $x \circ_e y = \J xey$ and $x^{*_e} = \J exe,$ respectively. The self-adjoint or hermitian part of $E_2(e)$ will be denoted by $E^{1}(e)$, that is, $$E^{1}(e)=\{x\in E_2(e) : x^{*_e} = \{e,x,e\}= x \}=\{x\in E : \{e,x,e\}= x \}.$$

Let us recall next the analogue to minimal partial isometry in the wider setting of JB$^*$-triples. A non-zero tripotent $e$ in a JB$^*$-triple $E$ is called (\emph{algebraically}) \emph{minimal} if  $E_2(e)=\CC e \neq \{0\}$. We shall denote by $\mathcal{U}_{min} (E)$ the set of all minimal tripotents in $E$. The tripotents $e\in E$ satisfying $E_0(e) =\{0\}$ are called \emph{complete}.\smallskip

As we have commented in the introduction, von Neumann algebras identify with those C$^*$-algebras which are dual Banach spaces. JBW$^*$-triples, defined as those JB$^*$-triples which are dual Banach spaces, play the role of von Neumann algebras in the JB$^*$-triple setting.  A concrete subclass is determined by those JBW$^*$-triples which coincide with the w$^*$-closure of the linear span of their minimal tripotents --$B(H)$ is an example--. The triples in this subclass are known as \emph{atomic} JBW$^*$-triples. Deep structure results, established by Y. Friedman and B. Russo, prove that every atomic JBW$^*$-triple is an $\ell_{\infty}$-sum of Cartan factors (cf. \cite[Proposition 2 and Theorem E]{FriRu86}), and that every JB$^*$-triple embeds isometrically as a JB$^*$-subtriple of an atomic JBW$^*$-triple.\smallskip

It is now moment to concrete the definition of the partial order and the notion of orthogonality among tripotents in a JB$^*$-triple. Let us take $e, v\in \mathcal{U} (E)$, where $E$ is a generic JB$^*$-triple. Following a concept that generalises the notion of orthogonality for partial isometries in $B(H)$, we shall say that $e$ is \emph{orthogonal} to $u$ ($e\perp u$ in short) if $\{e,e,u\}=0$ (equivalently,  $L(e,u) = 0\Leftrightarrow$ $L(u,e) = 0\Leftrightarrow$  $e \in E_0(u)\Leftrightarrow$ $u\in E_0(e)$ cf. \cite{loos1977bounded,Batt91,BurFerGarMarPe08}). It is known that any two orthogonal tripotents $e$ and $v$ in JB$^*$-triple $E$ are $M$-orthogonal, that is, $\|e\pm v\| = \max\{\|e\|, \|v\|\} = 1$ (cf. \cite[Lemma 1.3$(a)$]{FriRu85}). \smallskip

The rank of a JB$^*$-triple $E$ is the minimal cardinal number $r$ satisfying $\hbox{card}(S) \leq r$ for every \emph{orthogonal} subset $S \subseteq E$, where by an orthogonal subset we mean a subset not containing zero and satisfying that $x \perp y$ for every $x\neq y$ in $S$ (cf. \cite{Ka97}, \cite{BuChu92} and \cite{BeLoPeRo04} for basic background on the rank of a Cartan factor and a JBW$^*$-triple, and its relation with reflexivity).\smallskip

The natural partial order among partial isometries in $B(H)$ and, more generally, among tripotents in a JB$^*$-triple $E$ is defined by $e\leq u$ in $\mathcal{U} (E)$ if $u-e$ is a tripotent and $u-e \perp e$. This partial order is precisely the order consider by L. Moln{\'a}r in \cite{Molnar2002},  and it is a central notion in the theory of JB$^*$-triples (cf., for example, the recent papers \cite{HamKalPePfi20,HamKalPePfi20Groth,HamKalPe20, HamKalPe22determinants, Ham21, HamKalPeOrder}). %The partial order in $\mathcal{U} (E)$ enjoys several interesting properties; for example, $e\leq u$ if and only if $e$ is a projection in the JB$^*$-algebras $E_2(e)$ (cf. \cite[Lemma 3.2]{Batt91} or \cite[Corollary 1.7]{FriRu85} or \cite[Proposition 2.4]{HamKalPe20}). In particular, if $e$ and $p$ are tripotents (i.e. partial isometries) in a C$^*$-algebra $A$ regarded as a JB$^*$-triple with the triple product in \eqref{eq triple product JCstar triple} and $p$ is a projection, the condition $e\leq p$ implies that $e$ is a projection in $A$ with $e\leq p$ in the usual order on projections (i.e. $p e = e$).\smallskip
Thanks to the partial ordering we can consider tripotents which are minimal with respect to this ordering. It is easy to check that every algebraically minimal tripotent is (order) minimal, though the reciprocal implication does not necessarily hold for general JB$^*$-triples, in a JBW$^*$-triple order minimal and algebraic minimal tripotents coincide (cf. \cite[Corollary 4.8]{EdRutt88} and \cite[Lemma 4.7]{Batt91}).\smallskip

A very useful tool in the representation theory of JB$^*$-triples employed in our arguments and obtained in \cite[Lemma 3.10]{FerPe18Adv}, allows to describe the theoretical position of two arbitrary minimal tripotents in a Cartan factor of rank greater than or equal to 2. In order understand the statement employed later, we recall several basic relations between tripotents. Let $u,v$ be two tripotents in a JB$^*$-triple $E$. We shall say that $u$ and $v$ are \emph{collinear} ($u\top v$ in short) if $u\in E_1(v)$ and $v\in E_1(u)$. The tripotent $u$ \emph{governs} the tripotent $v$ ($u \vdash v$ in short) whenever $v\in E_{2} (u)$ and $u\in E_{1} (v)$. An ordered quadruple $(u_{1},u_{2},u_{3},u_{4})$ of minimal tripotents in a JB$^*$-triple $E$ is called a \emph{quadrangle} if $u_{1}\bot u_{3}$, $u_{2}\bot u_{4}$, $u_{1}\top u_{2} \top u_{3}\top u_{4} \top u_{1}$ and $u_{4}=2 \{{u_{1}},{u_{2}},{u_{3}}\}$ (the latter equality also holds if the indices are cyclically permutated, e.g. $u_{2} = 2 \{{u_{3}},{u_{4}},{u_{1}}\}$). %A \emph{prequadrangle} is an ordered set $(u_1,u_2,u_3)$ of three tripotents such that $u_1\top u_2\top u_3$ and $u_1\perp u_3$. 
An ordered triplet $ (v,u,\tilde v)$ of minimal tripotents in $E$, is called a \emph{trangle} if $v\bot \tilde v$, $u\vdash v$, $u\vdash \tilde v$ and $ v = Q(u)\tilde v$ (see \cite[\S 1]{DanFri87}).\smallskip

Along this note, the unit sphere and the closed unit ball of a normed space $X$ will be denoted by $S_{_X}$ and $\mathcal{B}_{X}$, respectively, and we shall write $\mathbb{T}$ for $S_{_\mathbb{C}}$.

\section{Main result}

In our first result, we shall see that bijections between sets of minimal tripotents in two atomic JBW$^*$-triples preserves the relation ``being collinear'' among them. It should be noted that Cartan factors of rank-one constitute a serious obstacle for the theorem describing the bijections preserving the partial order in both directions and orthogonality in one direction between the posets of all tripotents of two atomic JBW$^*$-triples (cf. \cite[Theorem 6.1 and Remark 3.6]{FriPe2021}), but in all the results in this manuscript we do not need to impose any restriction on the rank.\smallskip

In general, the linear span of all minimal tripotents in a JB$^*$-triple $E$, called the \emph{socle} of $E$ (soc$(E)$ in short), need not be a closed subspace. That is the case of the socle of $B(H)$, which coincides with the subspace, $\mathcal{F} (H)$, of all finite rank operators, and it is not closed when $H$ is infinite dimensional. However, if a JB$^*$-triple $E$ has finite rank (equivalently, $E$ is a reflexive JB$^*$-triple), we have soc$(E) = E$ (cf. \cite[Proposition 4.5 and Remark 4.6]{BuChu92} or \cite{BeLoPeRo04}).

\begin{proposition}\label{p preservation of collienarity and Peirce-1 subspaces} Let $\Phi: \mathcal{U}_{min}(M)\to \mathcal{U}_{min}(N)$ be a bijection preserving triple transition pseudo-probabilities, where $M$ and $N$ are two atomic JBW$^*$-triples. Suppose $e$ and $v$ are two minimal collinear {\rm(}$e \top v${\rm)} tripotents in $\mathcal{U}_{min}(M)$. Then $\Phi (e)$ and $\Phi(v)$ are collinear {\rm(}$\Phi (e)\top \Phi(v)${\rm)} in $\mathcal{U}_{min}(N)$.
\end{proposition}

\begin{proof} By hypotheses, $M$ and $N$ can be written as $\ell_{\infty}$-sums of two families of Cartan factors $\{C_i: i \in\Lambda_1\}$ and $\{D_j: j \in\Lambda_2\}$, respectively. Each minimal tripotent in $M$ (respectively, in $N$) lies in a single summand. Let us observe that, by hypotheses, $e$ and $v$ must lie in the same Cartan factor $C_{i_0}$ among those summands in $M$, otherwise they would be orthogonal.\smallskip
			
By \cite[Theorem 2.3]{Pe2022TTP} there exists a complex linear bijection $T_0: soc(M)\to soc(N)$ whose restriction to $\mathcal{U}_{min}(M)$ is $\Phi$.\smallskip

If $\Phi (e)$ and $\Phi(v)$ belong to different Cartan factors $D_{j_1}$ and $D_{j_2}$ with $j_1\neq j_2$, then they are orthogonal. However, in such a case $\frac{1}{\sqrt{2}} e + \frac{1}{\sqrt{2}} v$ is a minimal tripotent (cf. \cite[Lemma in page 306]{DanFri87}), and thus $\Phi \left(\frac{1}{\sqrt{2}} e + \frac{1}{\sqrt{2}} v \right) =  T_0 \left(\frac{1}{\sqrt{2}} e + \frac{1}{\sqrt{2}} v \right) = \frac{1}{\sqrt{2}} \Phi (e) + \frac{1}{\sqrt{2}} \Phi (v )$ must be a minimal tripotent too, which is incompatible with $\Phi(e)\perp \Phi(v)$. We can therefore assume that $\Phi (e)$ and $\Phi(v)$ both belong to the same Cartan factor $D_{j_0}$ and $\Phi(e) \not\perp \Phi(v)$. \smallskip

We shall distinguish two cases. If $D_{j_0}$ is a rank-one Cartan factor it must be a complex Hilbert space with inner product $\langle., .\rangle$, regarded as as type 1 Cartan factor, and both $\Phi(e)$ and $\Phi(v)$ are norm-one elements. As we commented before, they are collinear as minimal tripotents if and only if they are orthogonal in the Euclidean sense of this complex Hilbert space. Since, for each $(\lambda_1,\lambda_2)\in S_{\ell_2^2}$, the element $\lambda_1 e + \lambda_2 v$ is a minimal tripotent (cf. \cite[Lemma in page 306]{DanFri87}), its image under $\Phi$ or $T_0$, that is,  $\Phi(\lambda_1 e + \lambda_2 v)= T_0( \lambda_1 e + \lambda_2 v ) = \lambda_1 \Phi(e) + \lambda_2 \Phi(v)$, is a minimal tripotent in $D_{j_0}$, equivalently, a norm-one element of this Hilbert space. Therefore $$\begin{aligned}
1=\left\| \lambda_1 \Phi(e) + \lambda_2 \Phi(v) \right\|^2 &= \left\langle \lambda_1 \Phi(e) + \lambda_2 \Phi(v), \lambda_1 \Phi(e) + \lambda_2 \Phi(v)
\right\rangle \\
 &= |\lambda_1|^2 \|\Phi(e)\|^2 + |\lambda_2|^2  \|\Phi(v)\|^2 +2 \Re\hbox{e} \lambda_1 \overline{\lambda_2} \left\langle \Phi(e) ,  \Phi(v)
 \right\rangle \\
 &= 1 +2 \Re\hbox{e} \lambda_1 \overline{\lambda_2} \left\langle \Phi(e) ,  \Phi(v)
 \right\rangle
\end{aligned},$$ for all $(\lambda_1,\lambda_2)\in S_{\ell_2^2},$ witnessing that $\left\langle \Phi(e) ,  \Phi(v)
\right\rangle=0$, which proves that $\Phi(e)$ and $\Phi(v)$ are orthogonal in the Euclidean sense in the Hilbert space $D_{j_0}$, equivalently, collinear in the Cartan factor $D_{j_0}$.\smallskip

We assume next that $D_{j_0}$ is Cartan factor with rank $\geq 2$. By the representation result in \cite[Lemma 3.10]{FerPe18Adv}, applied to $\Phi(e)$ and $\Phi(v)$ in $D_{j_0}$, one of the following statements holds:\begin{enumerate}[$(1)$]\item There exist minimal tripotents $\tilde{v}_2,\tilde{v}_3,\tilde{v}_4$ in $D_{j_0}$ and $\alpha, \beta,\gamma, \delta\in \mathbb{C}$ such that $(\Phi(e),\tilde{v}_2,\tilde{v}_3,\tilde{v}_4)$ is a quadrangle and $\Phi(v) = \alpha \Phi(e) + \beta \tilde{v}_2 + \gamma \tilde{v}_4 + \delta \tilde{v}_3$, $\alpha \delta = \beta \gamma$ and $|\alpha|^2 + |\beta|^2 + |\gamma|^2 +|\delta|^2 =1$;
\item There exists a rank two tripotent $w\in D_{j_0},$ a minimal tripotent $\tilde{v}\in D_{j_0}$ and $\alpha, \beta, \gamma\in \mathbb{C}$ such that $(\Phi(e), w, \tilde{v})$ is a trangle, $\Phi(v) = \alpha \Phi(e) + \beta w + \delta \tilde{v}$, $\alpha \delta = \beta^2$ and $|\alpha|^2 + 2 |\beta|^2 +|\delta|^2 =1$.
\end{enumerate} Let us observe that $|\delta|<1,$ otherwise we would contradict $\Phi(e)\perp \Phi(v)$. We shall treat each case independently. \smallskip

$(1)$ Since $e$ and $v$ are collinear, the element $ \frac{1}{\sqrt{2}} e + \frac{1}{\sqrt{2}} v$ is a minimal tripotent (cf. \cite[Lemma in page 306]{DanFri87}) and the same must occur to $$\begin{aligned}\Phi \left( \frac{1}{\sqrt{2}} e + \frac{1}{\sqrt{2}} v \right) &= T_0 \left( \frac{1}{\sqrt{2}} e + \frac{1}{\sqrt{2}} v \right) =   \frac{1}{\sqrt{2}} \Phi \left( e\right) + \frac{1}{\sqrt{2}} \Phi \left(v \right) \\
	&= \frac{1}{\sqrt{2}} ( 1+\alpha) \Phi(e) + \frac{1}{\sqrt{2}} \beta \tilde{v}_2 + \frac{1}{\sqrt{2}} \gamma \tilde{v}_4 + \frac{1}{\sqrt{2}} \delta \tilde{v}_3,
\end{aligned}$$ but the latter being a minimal tripotent implies that  $(1+\alpha)\delta - \beta \gamma =0$. However, since $\alpha \delta = \beta \gamma$, it follows that $\delta =0 = \beta \gamma$. Let us assume that $\beta =0$ (the other case, i.e. $\gamma=0$, is similar). Therefore, the element $\Phi \left( \frac{1}{\sqrt{2}} e + \frac{1}{\sqrt{2}} v \right) = \frac{1}{\sqrt{2}} ( 1+\alpha) \Phi(e) + \frac{1}{\sqrt{2}} \gamma \tilde{v}_4 $ must be a minimal tripotent, what occurs if and only if  $\frac{| 1 + \alpha|^2}{2} + \frac{|\gamma|^2}{2} =1$. We deduce that $\Re\hbox{e}(\alpha)=0$. Replacing $e$ with $i e$, and having in mind that $i e $ and $v$ are collinear too and $T_0$ is complex linear, we get $\Im\hbox{m} (\alpha) =0$. Therefore, $\alpha =0,$ $\gamma\in \mathbb{T}$ and $\Phi(v) = \gamma \tilde{v}_4$ is collinear to $\Phi(e)$. \smallskip

$(2)$  As in the previous case, the element $$\Phi \left( \frac{1}{\sqrt{2}} e + \frac{1}{\sqrt{2}} v \right) = T_0 \left( \frac{1}{\sqrt{2}} e + \frac{1}{\sqrt{2}} v \right) = \frac{1}{\sqrt{2}} (1+\alpha) \Phi(e) + \frac{1}{\sqrt{2}} \beta w + \frac{1}{\sqrt{2}} \delta \tilde{v},$$ must be a minimal tripotent, and thus $(1+\alpha)\delta - \beta^2 =0,$ and consequently $\delta = \beta =0$ and $\alpha \in \mathbb{T}$ which is impossible because $\Phi$ preserves triple transition pseudo-probabilities. So, the second case is discarded and the proof is concluded.
\end{proof}

We can now establish our main result showing that every bijection preserving triple transition pseudo-probabilities between the posets of minimal tripotents of two atomic JBW$^*$-triples automatically preserves orthogonality.

\begin{theorem}\label{t preservers of ttp also preserve orthog} Let $\Phi: \mathcal{U}_{min}(M)\to \mathcal{U}_{min}(N)$ be a bijection preserving triple transition pseudo-probabilities, where $M$ and $N$ are two atomic JBW$^*$-triples. Then $\Phi$ preserves orthogonality in both directions.
\end{theorem}

\begin{proof} We begin by observing that $\Phi^{-1}$ also preserves triple transition pseudo-probabilities. Let  $T_0: soc(M)\to soc(N)$ be the bijection extending $\Phi$ whose existence is given by \cite[Theorem 2.3]{Pe2022TTP}.\smallskip
	
Let us take $e,v\in \mathcal{U}_{min}(M)$ with $e\perp v$. We shall prove that $\Phi(e)\perp \Phi(v)$. By hypotheses, $\displaystyle M= \bigoplus_{i\in \Lambda_1}^{\ell_{\infty}} C_i$ and  $\displaystyle N = \bigoplus_{j\in \Lambda_2}^{\ell_{\infty}} {D}_j,$ where $C_i$ and $D_j$ are Cartan factors.\smallskip

If $\Phi (e)$ and $\Phi(v)$ belong to different Cartan factors $D_{j_1}$ and $D_{j_2}$ with $j_1\neq j_2$, the desired conclusion is clear. We shall therefore assume that $\Phi (e), \Phi(v)\in D_{j_0}.$\smallskip

If $D_{j_0}$ has rank-one, it must be a complex Hilbert space regarded as a type 1 Cartan factor, and both elements $\Phi(e)$ and $\Phi(v)$ are in its unit sphere. If dim$(D_{j_0}) =1$ (i.e. $D_{j_0}= \mathbb{C}$) $\Phi (v) = \mu \Phi (e)$ for some unitary $\mu \in \mathbb{C}$. However, $0 = TTP(e,v)  = TTP ( \Phi (e), \Phi(v)) = \mu$, which is impossible. We can therefore assume that dim$(D_{j_0}) \geq 2,$ and find a third tripotent $\widehat{w}\in D_{j_0}$ (i.e. an element in the unit sphere of $D_{j_0}$) and $(\lambda_1,\lambda_2)\in S_{\ell_2^2}$ such that $\Phi(e) \perp_2 \widehat{w}$ in the Euclidean sense and $\Phi (v) = \lambda_1 \Phi(e) + \lambda_2 \widehat{w}$. By applying $\Phi^{-1}$ and $T_0^{-1}$ we derive that $v = \lambda_1 e + \lambda_2 \Phi^{-1} (\widehat{w})$. Since $\Phi(e) \top \widehat{w}$ in $D_{j_0}$ (and hence in $N$), Proposition \ref{p preservation of collienarity and Peirce-1 subspaces}, applied to $\Phi^{-1}$, implies that $e$ and  $\Phi^{-1}(\widehat{w})$ are collinear, which contradicts the fact that $v\perp e$, because $0=\{e,e,v\} = \{e,e,\lambda_1 e + \lambda_2 \Phi^{-1} (\widehat{w})\} = \lambda_1 e + \frac{\lambda_2}{2} \Phi^{-1} (\widehat{w}),$ and thus $\lambda_1 = \lambda_2=0$. Therefore $D_{j_0}$ must have rank $\geq 2$.\smallskip

Since $D_{j_0}$ is Cartan factor with rank $\geq 2$, Lemma 3.10 in \cite{FerPe18Adv}, applied to $\Phi(e)$ and $\Phi(v)$ in $D_{j_0}$, assures that one of the following statements holds:\begin{enumerate}[$(1)$]\item There exist minimal tripotents $\tilde{v}_2,\tilde{v}_3,\tilde{v}_4$ in $D_{j_0}$ and $\alpha, \beta,\gamma, \delta\in \mathbb{C}$ such that $(\Phi(e),\tilde{v}_2,\tilde{v}_3,\tilde{v}_4)$ is a quadrangle and $\Phi(v) = \alpha \Phi(e) + \beta \tilde{v}_2 + \gamma \tilde{v}_4 + \delta \tilde{v}_3$, $\alpha \delta = \beta \gamma$ and $|\alpha|^2 + |\beta|^2 + |\gamma|^2 +|\delta|^2 =1$;
	\item There exists a rank two tripotent $w\in D_{j_0},$ a minimal tripotent $\tilde{v}\in D_{j_0}$ and $\alpha, \beta, \gamma\in \mathbb{C}$ such that $(\Phi(e), w, \tilde{v})$ is a trangle, $\Phi(v) = \alpha \Phi(e) + \beta w + \delta \tilde{v}$, $\alpha \delta = \beta^2$ and $|\alpha|^2 + 2 |\beta|^2 +|\delta|^2 =1$.
\end{enumerate}

We treat both cases in parallel. Since $\Phi$ preserves triple transition pseudo-probabilities, $0=TTP(e,v)= TTP (\Phi(e),\Phi(v)) = \alpha$. This implies in case $(2)$ that $\beta =0$ and $\Phi(v) = \delta \tilde{v} \perp \Phi(e)$, which gives the desired conclusion.\smallskip

We finally handle case $(1)$. Since $\alpha=0= \beta \gamma$ one of these two scalars is zero. We can assume that $\beta=0$ (the other case is similar). Then $\Phi(v) = \gamma \tilde{v}_4 + \delta \tilde{v}_3$ with $|\gamma|^2 +|\delta|^2 =1$. Since $\tilde{v}_4\top \Phi(e)$ and $\tilde{v}_4\top \tilde{v}_3$, Proposition \ref{p preservation of collienarity and Peirce-1 subspaces} assures that $e \top \Phi^{-1}(\tilde{v}_4)$ and $\Phi^{-1} (\tilde{v}_4)\top \Phi^{-1} (\tilde{v}_3)$. Taking images under $\Phi^{-1}$ and $T_0^{-1}$ we get \begin{equation}\label{eq preimages of the quadrangle by Phi inverse} 0=\{e,e,v\} = \{e,e, \gamma \Phi^{-1} (\tilde{v}_4) + \delta \Phi^{-1}(\tilde{v}_3)\} = \frac{\gamma}{2} \Phi^{-1} (\tilde{v}_4) + \delta\  \{e,e,\Phi^{-1}(\tilde{v}_3)\}.
\end{equation} Let us make a couple of observations. First, by the preservation of triple transition pseudo-probabilities we have
$$0 = TTP (\Phi(e), \tilde{v}_3) = TTP (e , \Phi^{-1} (\tilde{v}_3)),$$ which implies that $P_2(\Phi^{-1} (\tilde{v}_3)) (e) = 0$, and consequently $e =  P_1(\Phi^{-1} (\tilde{v}_3)) (e) + P_0(\Phi^{-1} (\tilde{v}_3)) (e)$. Therefore, by Peirce arithmetic,
$$\begin{aligned}
\{e,e,\Phi^{-1}(\tilde{v}_3)\} &= \{P_0(\Phi^{-1} (\tilde{v}_3)) (e),P_1(\Phi^{-1} (\tilde{v}_3)) (e),\Phi^{-1}(\tilde{v}_3)\} \\
&+ \{P_1(\Phi^{-1} (\tilde{v}_3)) (e),P_1(\Phi^{-1} (\tilde{v}_3)) (e),\Phi^{-1}(\tilde{v}_3)\},	
\end{aligned}$$ where $\{P_0(\Phi^{-1} (\tilde{v}_3)) (e),P_1(\Phi^{-1} (\tilde{v}_3)) (e),\Phi^{-1}(\tilde{v}_3)\} \in M_1 (\Phi^{-1}(\tilde{v}_3)),$ while the second summand
$\{P_1(\Phi^{-1} (\tilde{v}_3)) (e),P_1(\Phi^{-1} (\tilde{v}_3)) (e),\Phi^{-1}(\tilde{v}_3)\} \in M_2 (\Phi^{-1}(\tilde{v}_3)).$ Having in mind that $\Phi^{-1} (\tilde{v}_4)\top \Phi^{-1} (\tilde{v}_3)$, and hence $\Phi^{-1} (\tilde{v}_4)\in M_{1} (\Phi^{-1} (\tilde{v}_3))$, we deduce from \eqref{eq preimages of the quadrangle by Phi inverse} and the Peirce decomposition with respect to $\Phi^{-1} (\tilde{v}_3)$ that \begin{equation}\label{eq last equation in thm} \delta\ \{P_1(\Phi^{-1} (\tilde{v}_3)) (e),P_1(\Phi^{-1} (\tilde{v}_3)) (e),\Phi^{-1}(\tilde{v}_3)\}=0.
\end{equation} If $\delta=0$, it follows from \eqref{eq preimages of the quadrangle by Phi inverse} that $\gamma  =0,$ and hence $\Phi(v) =0$, which is impossible.\smallskip

The other alternative from \eqref{eq last equation in thm} is $\{P_1(\Phi^{-1} (\tilde{v}_3)) (e),P_1(\Phi^{-1} (\tilde{v}_3)) (e),\Phi^{-1}(\tilde{v}_3)\}=0.$ Now, an application of \cite[Lemma 1.5]{FriRu85} or \cite[Theorem 2.3]{Pe2015} gives $P_1(\Phi^{-1} (\tilde{v}_3)) (e)=0$, and thus $e = P_0(\Phi^{-1} (\tilde{v}_3)) (e) \perp \Phi^{-1} (\tilde{v}_3)$. Finally, by the above arguments, we have $$v = \Phi^{-1} (\gamma \tilde{v}_4 + \delta \tilde{v}_3) = T_0^{-1} (\gamma \tilde{v}_4 + \delta \tilde{v}_3) = \gamma \Phi^{-1}(\tilde{v}_4) + \delta \Phi^{-1} (\tilde{v}_3),$$ with $v\perp e$, $e \perp \Phi^{-1} (\tilde{v}_3)$ and $\Phi^{-1} (\tilde{v}_4)\top e$ (compare Proposition \ref{p preservation of collienarity and Peirce-1 subspaces}), it necessarily holds that $\gamma =0,$ and thus $\Phi (v) = \delta \tilde{v}_3 \perp \Phi(e).$
\end{proof}

By combining Theorem \ref{t preservers of ttp also preserve orthog} with \cite[Corollary 2.5]{Pe2022TTP} we appreciate the real impact of our conclusions in the next corollary.

\begin{corollary}\label{c bijections preserving triple transition pseudo-probabilities and orthogonality} Let $\Phi : \mathcal{U}_{min}(M) \to \mathcal{U}_{min}(N)$ be a bijective transformation preserving triple transition pseudo-probabilities {\rm(}i.e. $TTP(\Phi(v),\Phi(e))=\varphi_{\Phi(e)} (\Phi(v)) = \varphi_{e} (v)=TTP(v,e),$ for all $e,v$ in $\mathcal{U}_{min} (M)${\rm)}, where $M$ and $N$ are atomic JBW$^*$-triples. Then $\Phi$ extends {\rm(}uniquely{\rm)} to a surjective complex-linear {\rm(}isometric{\rm)} triple isomorphism from $M$ onto $N$.
\end{corollary}

The next corollary is perhaps in an interesting surprise by itself. 

\begin{corollary}\label{c bijections preserving triple transition pseudo-probabilities are isometries} Let $\Phi : \mathcal{U}_{min}(M) \to \mathcal{U}_{min}(N)$ be a bijective transformation preserving triple transition pseudo-probabilities, where $M$ and $N$ are atomic JBW$^*$-triples. Then $\Phi$ is an isometry with respect to the distances given by the triple norms. 
\end{corollary}

All previous results also hold for atomic von Neumann algebras (i.e. $\ell_{\infty}$-sums of $B(H)$ spaces).\smallskip

Let $M$ and $N$ be atomic JBW$^*$-triples. Under the light of Corollary \ref{c bijections preserving triple transition pseudo-probabilities are isometries} above, it is natural to ask whether a bijection $\Phi:  \mathcal{U}_{min}(M) \to \mathcal{U}_{min}(N)$ preserving distances with respect to the triple norms also preserves triple transition pseudo-probabilities. The answer is, in general, negative. The counterexamples presented below points out the different information encoded by the set of minimal tripotents equipped with the triple transition pseudo-probability and the set of minimal projections with the usual transition probability. For example, the natural conjugation on $B(H),$ $a \mapsto a^*$, defines a conjugate-linear (isometric) triple automorphism whose restriction to $\mathcal{U}_{min}(B(H))$, defines a bijection  $\Psi:  \mathcal{U}_{min}(B(H)) \to \mathcal{U}_{min}(B(H))$ which preserves distances, but does not preserve triple transition pseudo-probabilities since $TTP(\Psi(\lambda e),\Psi(e)) =  TTP(\overline{\lambda} e^*,e^*) = \overline{\lambda}$ is not, in general, equal to $TTP(\lambda e,e) =\lambda $ for all $e
\in \mathcal{U}_{min}(B(H))$ and $\lambda\in \mathbb{T}.$  

\begin{remark}\label{r TTP and distance are not mutually determined} The usual operator or C$^*$- norm on $B(H)$ induces a metric on the set $\mathcal{P}_1 (H),$ of all minimal projections in $B(H)$, which is known as the \emph{gap metric}. The gap metric and the transition probability between elements in  $\mathcal{P}_1 (H)$ are mutually determined by the formula \begin{equation}\label{eq gap metric determined by the transition probability} \| p - q \| = \sqrt{ 1 - \hbox{tr} (pq)} = \sqrt{ 1 - TTP (p,q)},
	\end{equation} {\rm(}cf. \cite[$(2.6.13)$ in page 127]{Molnar85}{\rm)}. The distance, or gap metric, between two minimal tripotents $e$ and $v$ in a JB$^*$-triple $E$ was determined in \cite[Proposition 3.3]{FerPe18Adv} and can be computed with the following formula:
	\begin{equation}\label{eq formula for the gap metric between minimal tripotents} \|e-v\|^2 = (1-\Re\hbox{e} TTP(v,e)) + \sqrt{(1-\Re\hbox{e} TTP(v,e))^2 -\|P_0(e) (v)\|^2}.
	\end{equation} It does not take too much time to check that \eqref{eq formula for the gap metric between minimal tripotents} coincides with the formula in \eqref{eq gap metric determined by the transition probability} when $e$ and $v$ are minimal projections {\rm(}i.e. positive minimal partial isometries{\rm)} in $B(H)$. To illustrate the statement that the gap metric and the triple transition pseudo-probability are not mutually determined with a concrete example, consider the tripotents $e = \left( \begin{matrix}
		1 & 0 \\
		0 & 0
	\end{matrix}\right),$ $v = \left( \begin{matrix}
		1/3 & 1/3 \\
		\sqrt{{7}/{18}} & \sqrt{{7}/{18}}
	\end{matrix}\right)$ and $\tilde{v} = \left( \begin{matrix}
		1/3 & 1/4 \\
		{\sqrt{119}}/{15} & {\sqrt{119}}/{20}
	\end{matrix}\right)$ in $M_2(\mathbb{C})$. It is easy to check that $TTP(v,e) = TTP(\tilde{v},e) = 1/3$ while $\|e-v\| \neq \|e- \tilde{v}\|$. On the other hand, taking $\gamma,\beta\in \mathbb{R}$ such that $\gamma \beta = 1/2  \left({\sqrt{3-\sqrt{2}}}\right)/(3 \sqrt{2})$, $1/4 + \beta^2 +\gamma^2 + {{(3-\sqrt{2})}}/{18} =1$, the element $u = \left( \begin{matrix}
		1/2 &  \beta\\
		\gamma & \left({\sqrt{3-\sqrt{2}}}\right)/(3 \sqrt{2})
	\end{matrix}\right)$ is a minimal tripotent satisfying $\|e-v\|^2 = \frac{1+ 2 \sqrt{2}}{3 \sqrt{2}} = \|e -u\|^2$ while $TTP(v,e) = 1/3 \neq 1/2 = TTP(u,e)$.
\end{remark}

The previous discussion naturally leads to the study of surjective isometries between sets of minimal tripotents in two atomic JBW$^*$-triples. We are therefore connected with the celebrated Tingley's problem in the case of atomic JBW$^*$-triples \cite{FerPe17c,FerPe18}. The main result in \cite{FerPe17c} shows that every surjective isometry  $\Delta$ between the unit spheres of two atomic JBW$^*$-triples $M$ and $N$ admits a extension to a real linear triple isomorphism between the JB$^*$-triples. Clearly, the set of minimal tripotents in a JB$^*$-triple $E$ is contained in the unit sphere of $E$. In the case of a complex Hilbert space, regarded as a type 1 Cartan factor, the set of all minimal tripotents is precisely the whole sphere.  One of the key facts in the just commented result from \cite{FerPe17c} consists in proving that such an isometry $\Delta$ maps $\mathcal{U}_{min}(M)$ to $\mathcal{U}_{min}(N)$. In our next result this will be part of the hypothesis but at the cost of reducing the domain of our bijection.\smallskip

We recall some terminology first. Following \cite{FerMarPe2012}, the set of all \emph{contractive perturbations} of a subset $S$ of the closed unit ball of a Banach space $X$ is defined as the 
norm-closed convex subset of $\mathcal{B}_{X}$ given by $$\hbox{\rm cp}(S) =\{ x\in X : \|x \pm s \| \leq 1 \}.$$ For each natural $n\geq 2$, the $n$-th contractive perturbations of $S$ are inductively defined by the equality $\hbox{\rm cp}^{(n)} (S) = \hbox{\rm cp}(\hbox{\rm	cp}^{(n-1)} (S))$. It is known that $S\subseteq \hbox{\rm cp}^{(2)} (S)$, which gives $\hbox{\rm cp} (S)= \hbox{\rm cp}^{(3)} (S).$\smallskip
  
One of the basic tools in our previous arguments is provided by \cite[Lemma 3.10]{FerPe18Adv}, a result which describes the relative position of two minimal tripotents in a Cartan factor of rank $\geq 2$. We shall next state an analogous result for rank-one Cartan factors, which has been employed before. The statement is probably part of the folklore in JB$^*$-triple theory and the proof is clear. 

\begin{lemma}\label{l lemma 3.10 in advances for rank-one Cartan factors} Let $e,v$ be two minimal tripotents in a rank-one Cartan factor $C$ with dimension $\geq 2$. Then there exists another minimal tripotent $v_1$ in $C$ and $\alpha, \beta\in \mathbb{C}$ satisfying $$e\top v_1, \ |\alpha|^2+ |\beta|^2 =1, \hbox{ and } v =  \alpha e + \beta v_1.$$    \end{lemma} 

%\begin{proof} It suffices to note that $C$ must be a complex Hilbert space regarded as a type 1 Cartan factor and $e$ and $v$ are two norm-one elements. By the hypothesis concerning the dimension of $C$, we can find another element $v_1\in C,$ $\alpha, \beta\in \mathbb{C}$ such that $\{e,v_1 \}$ is an orthonormal system of the Hilbert space and $ v =  \alpha e + \beta v_1.$ It is easy to check that $e\top v_1$. \end{proof}

\begin{remark}\label{r surjective real linear isometries} Each surjective real linear isometry between two Cartan factors of rank $\geq 2$ must be either complex linear or conjugate linear and a triple isomorphism {\rm(}cf. \cite{Dang92}{\rm)}. A similar conclusion is, in general, false for rank-one Cartan factors. Namely, the mapping $T_0:\ell_2^2\to \ell_2^2,$ $T_0( \lambda_1, \lambda_2) = ( \lambda_1, \overline{\lambda_2})$ is a surjective real linear isometry which is not complex linear nor conjugate linear and does not preserve triple products. Let us see that this covers all possible possibilities.  Suppose $T: H\to H$ is a surjective real linear isometry between to rank-one Cartan factors {\rm(}i.e. two complex Hilbert spaces which are clearly identified{\rm)}. Let $\{e_j : \Lambda\}$ be an orthonormal basis of the Hilbert space $H$. It is easy to check that each $e_j$ defines a minimal and maximal tripotent in $H$ and they are all mutually collinear. It is also clear that $T$ preserves cubes and collinearity {\rm(}i.e. Euclidean orthogonality{\rm,)}. Therefore $\{T(e_j)\}$ is an orthonormal basis of $H$ too, and the equation $\|T(e_j) -T(i e_j) \|^2 = 2$ then implies that $T(i e_j) \in \{\pm i T(e_j)\}$ for all $j$. Setting $\Lambda_1:=\{j : T(i e_j) =i T(e_j) \},$ $\Lambda_2:=\{j : T(i e_j) =- i T(e_j) \},$ and $H_k =\overline{\hbox{span}}\{e_j: j\in \Lambda_k\},$ we have $H = H_1\oplus_2^{\perp_2} H_2$, $T|_{H_1}: H_1\to T(H_1)$ is a complex linear surjective isometry and $T|_{H_2}: H_2\to T(H_2)$ is a conjugate linear surjective isometry. Furthermore, the natural conjugation $j$ on $H$ defined by $j (k_1,k_2) :=  (k_1,\overline{k_2})$ {\rm(}$(k_1,k_2)\in T(H_1)\oplus_2^{\perp_2} T(H_2)${\rm)} satisfies that $j\circ T : H\to H$ is a complex linear isometry and a triple isomorphism. 
\end{remark}

Our next result determines the form of all surjective isometries between the sets of minimal tripotents in two atomic JBW$^*$-triples. 

\begin{theorem}\label{t Tingley for minimal in atomic JBW-triples} Let $\Delta: \mathcal{U}_{min}(M) \to \mathcal{U}_{min}(N)$ be a surjective isometry, where $M$ and $N$ are atomic JBW$^*$-triples. Then there exists a {\rm(}unique{\rm)} real linear isometry $T:M\to N$ such that $\Delta=T|_{\mathcal{U}_{min}(M)}$.
\end{theorem}

\begin{proof} The proof will be obtained by adequate adaptations of tools developed in the study of Tingley's problem for JB$^*$-triples (cf. \cite{PeTan19, FerPe18Adv, FerPe17c, FerPe18}) and the new conclusion in Corollary \ref{c bijections preserving triple transition pseudo-probabilities and orthogonality}. \smallskip
	
Let us begin by proving that $\Delta$ maps antipodal points to antipodal points, that is, \begin{equation}\label{eq Delta preserves antipodal} \Delta(-e) = -\Delta(e) \hbox{ for all } e\in \mathcal{U}_{min}(M).
\end{equation} Namely, since by hypothesis we have $\|\Delta(e) - \Delta(-e) \| = \|e - (-e)\| =2$, Proposition 2.2 in \cite{FerPe18Adv} assures that $$\Delta(-e) = -\Delta(e) + P_0(\Delta(e)) (\Delta(-e)).$$ However, the fact that $\Delta(-e)$ is a minimal tripotent implies that $\Delta(-e) = -\Delta(e)$ as desired.\smallskip

We shall next show that $\Delta$ preserves orthogonality among minimal tripotents in both directions, concretely,  \begin{equation}\label{eq Delta preserves orthogonality isometry} e\perp v \hbox{ in } \mathcal{U}_{min}(M) \Leftrightarrow \Delta(e) \perp \Delta(v) \hbox{ in } \mathcal{U}_{min}(N).
\end{equation} Let us take $e,v\in \mathcal{U}_{min}(M)$ with $e\perp v$. In this case, $\|e \pm v\| = 1,$ by orthogonality, and thus $\|\Delta(e) 
\pm \Delta(v)\| =1$ (cf. \eqref{eq Delta preserves antipodal}), assuring that $\Delta(v) \in cp(\Delta(e))$. As shown in \cite[$(6)$ in page 360]{FerMarPe2012}, $cp(\{\Delta(e)\}) = \{y\in \mathcal{B}_{N} : y
\perp \Delta(e) \} = \mathcal{B}_{N_0 (\Delta(e))}$, which proves that $\Delta(v)\in \mathcal{B}_{N_0 (\Delta(e))},$ and hence $\Delta(e)\perp \Delta(v)$. \smallskip

In the sequel, we shall apply that $M$ and $N$ are atomic JBW$^*$-triples, and hence we can write $\displaystyle M= \bigoplus_{i\in \Lambda_1}^{\ell_{\infty}} C_i$ and  $\displaystyle N = \bigoplus_{j\in \Lambda_2}^{\ell_{\infty}} {D}_j,$ where $C_i$ and $D_j$ are Cartan factors. Let us comment a basic fact. Each minimal tripotent in $M$ and in $N$ lies in a single summand of the corresponding decomposition. Therefore, $e\not\perp v$ in $\mathcal{U} (M)_{min}$ implies that $e$ and $v$ belong to the same Cartan factor $C_{i_0}$, and by \eqref{eq Delta preserves orthogonality isometry} $\Delta(e),\Delta(v)$ are contained in the same Cartan factor $D_{j_0}$. \smallskip

Our next goal consists in proving that for each $i\in \Lambda_1$ there exists a unique $\sigma(i)\in \Lambda_2$ such that $\Delta (\mathcal{U}_{min}(C_{i})) = \mathcal{U}_{min} (D_{\sigma(i)})$, and both summands $C_{i}$ and $D_{\sigma(i)}$ have the same rank. Namely, fix any $e\in \mathcal{U}_{min}(C_{i})$ and pick $\sigma(i)\in \Lambda_2$ such that $\Delta(e)\in \mathcal{U}_{min} (D_{\sigma(i)})$. Given any other $v\in \mathcal{U}_{min}(C_{i})$, by \cite[Lemma 3.10]{FerPe18Adv} and Lemma \ref{l lemma 3.10 in advances for rank-one Cartan factors} there exists $w\in \mathcal{U}_{min}(C_{i})$ such that $w\not\perp e,v$. It follows from the above comments that $\Delta(e),\Delta(v)$ and $\Delta(w)$ all lie in the same factor of the decomposition of $N$, therefore $\Delta(v)\in  \mathcal{U}_{min} (D_{\sigma(i)})$. This proves that $\Delta (\mathcal{U}_{min}(C_{i}))\subseteq  \mathcal{U}_{min} (D_{\sigma(i)})$, and the equality follows from the same argument applied to $\Delta^{-1}$. The rest is clear from the bijectivity of $\Delta $ and \eqref{eq Delta preserves orthogonality isometry}. \smallskip

Pick $i\in \Lambda_1$ such that $C_{i}$ and $D_{\sigma(i)}$ have rank-one. In this case $\Delta|_{\mathcal{U}_{min}(C_{i})} : \mathcal{U}_{min}(C_{i})\to  \mathcal{U}_{min} (D_{\sigma(i)})$ is a surjective isometry, and the assumption concerning the rank implies that $\mathcal{U}_{min}(C_{i}) = S_{C_{i}}$ and $\mathcal{U}_{min} (D_{\sigma(i)}) = S_{D_{\sigma(i)}}$. We can therefore apply Ding's solution to Tingley's problem for Hilbert spaces \cite[Theorem 2.2]{Ding2002} to deduce the existence of a surjective real linear isometry $T_i: C_{i}\to D_{\sigma(i)}$ satisfying $T_i (e) = \Delta (e)$ for all $e\in \mathcal{U}_{min}(C_{i}) = S_{C_{i}}$. Although, $T_i$ need not be complex linear nor conjugate linear (cf. Remark \ref{r surjective real linear isometries}), we can find a conjugation $j_i$ on $D_{\sigma(i)}$ such that $j_i \circ T_i : C_{i}\to D_{\sigma(i)}$ is an isometric (linear) triple isomorphism.\label{eq conjugation on rank-one} This concludes the discussion for rank-one Cartan factors in the decomposition of $M$.\smallskip
	
In the following we shall focus in the case in which $C_i$ and $ D_{\sigma(i)}$ are Cartan factors with rank $\geq 2$ and $\Delta|_{\mathcal{U}_{min}(C_{i})} : \mathcal{U}_{min}(C_{i})\to  \mathcal{U}_{min} (D_{\sigma(i)})$ is a surjective isometry. \smallskip

Let us next prove that \begin{equation}\label{eq Delta i times minimal} \Delta(i e) \in \{\pm i\Delta(e)\}, \hbox{ for all } e\in \mathcal{U}_{min}(C_{i}).
\end{equation} Namely, pick a minimal tripotent $\tilde{v}= \Delta(v)\perp \Delta(e)$ in  $\mathcal{U}_{min} (D_{\sigma(i)})$. By \eqref{eq Delta preserves orthogonality isometry} $v\perp e,$ equivalently, $v\perp i e,$ and thus $\tilde{v}= \Delta(v)\perp \Delta(i e)$ (cf. \eqref{eq Delta preserves orthogonality isometry}). It follows that $$\begin{aligned} & \{\Delta(e)\}^{\perp} \cap \mathcal{U}_{min} (D_{\sigma(i)})  =\{ \tilde{v}\in \mathcal{U}_{min} (D_{\sigma(i)}) : \tilde{v}\perp \Delta(e)\} \\ &= \{\Delta(i e)\}^{\perp} \cap \mathcal{U}_{min} (D_{\sigma(i)}) =\{ \tilde{v}\in \mathcal{U}_{min} (D_{\sigma(i)}) : \tilde{v}\perp \Delta(i e)\}. 
\end{aligned}$$ 

Since the linear combinations of minimal tripotents in the orthogonal complement of $\Delta (e)$ in $D_{\sigma(i)}$ are weak$^*$ dense in this orthogonal complement, we deduce from the above that $$\begin{aligned} & \{\Delta(e)\}^{\perp} \cap D_{\sigma(i)}  =\{ \tilde{z}\in D_{\sigma(i)} : \tilde{z}\perp \Delta(e)\} \\ &= \{\Delta(i e)\}^{\perp} \cap D_{\sigma(i)} =\{ \tilde{z}\in D_{\sigma(i)} : \tilde{z}\perp \Delta(i e)\}. 
\end{aligned}$$ Consequently, $$\Delta(i e)\in \{\Delta(e)\}^{\perp\perp } \cap D_{\sigma(i)}  =\{ \tilde{z}\in D_{\sigma(i)} : \tilde{z} \perp \{\Delta(e)\}^{\perp}  \}.$$ Since $\Delta(e)$ is a minimal tripotent in a Cartan factor with rank $\geq 2$, it cannot be complete, and hence $\{\Delta(e)\}^{\perp\perp } \cap D_{\sigma(i)}  = \left(D_{\sigma(i)}\right)_2(\Delta(e))  = \mathbb{C} \Delta(e)$ (cf. \cite[Remark 3.4]{BurGarPe11}). We can therefore find a unitary $\mu\in \mathbb{T}$ such that $\Delta(i e) = \mu \Delta(e)$. By applying that $\Delta$ is an isometry we get $$|1- \mu| = \|\Delta(e) - \Delta(i e)\| = \|e - ie \| = \sqrt{2},$$ which gives $\mu = \pm i,$ and concludes the proof of \eqref{eq Delta i times minimal}.\smallskip

Building upon \eqref{eq Delta i times minimal} we can now prove the key step in the proof. Let $e\in \mathcal{U}_{min}(C_{i}),$ then one, and precisely one, of the following statements holds: \begin{enumerate}
	\item[$({\color{red}14.a})$] $\Delta( i e ) = i \Delta (e)$, $TTP(v,e) = TTP(\Delta(v),\Delta(e))$ and $\Delta( i v ) = i \Delta (v)$ for all $v\in \mathcal{U}_{min}(C_{i})$.
	\item[$({\color{red}14.b})$] $\Delta( i e ) = - i \Delta (e)$, $TTP(v,e) = \overline{TTP(\Delta(v),\Delta(e))}$ and $\Delta( i v ) = - i \Delta (v)$ for all $v\in \mathcal{U}_{min}(C_{i})$.
\end{enumerate} Fix arbitrary elements $e,v\in \mathcal{U}_{min}(C_{i}).$ The proof relies on the relative position of the minimal tripotents $e$, $v$ and their images. By \cite[Lemma 3.10]{FerPe18Adv} applied to $e,v$ in $C_i$ and their images in $D_{\sigma(i)}$ one of the statements in the following two couples holds:\smallskip\medskip

\(
\left\{ \quad
\begin{minipage}[c]{0.93\linewidth}
	\item[\textbf{(1)}] There exist minimal tripotents $v_2,v_3,v_4$ in $C_{i}$ and complex numbers $\alpha$, $\beta$, $\gamma$, $\delta$ such that $(e,v_2,v_3,v_4)$ is a quadrangle, $|\alpha|^2 +| \beta|^2 + |\gamma|^2 + |\delta|^2 =1$, $\alpha \delta  = \beta \gamma$, and $v = \alpha e + \beta v_2 + \gamma v_4 + \delta v_3$;
	\item[\textbf{(2)}]  There exist a minimal tripotent $\tilde v\in C_{i}$, a rank two tripotent $u\in C_{i}$, and complex numbers $\alpha, \beta, \delta$ such that $(e, u,\tilde v)$ is a trangle, $|\alpha|^2 +2 | \beta|^2 + |\delta|^2 =1$, $\alpha \delta  = (\beta)^2$, and $v = \alpha e + \beta u +\delta \tilde v$.
\end{minipage}
\right.
\)\smallskip\medskip

\(
\left\{ \quad
\begin{minipage}[c]{0.93\linewidth}
	\item[\textbf{(1')}] There exist minimal tripotents $w_2,w_3,w_4$ in $D_{\sigma(i)}$, and complex numbers $\alpha'$, $\beta'$, $\gamma'$, $\delta'$ such that $(\Delta(e),w_2,w_3,w_4)$ is a quadrangle, $|\alpha'|^2 +| \beta'|^2 + |\gamma'|^2 + |\delta'|^2 =1$, $\alpha^{\prime} \delta^{\prime}  = \beta^{\prime} \gamma^{\prime}$, and $\Delta(v) = \alpha' \Delta (e) + \beta' w_2 + \gamma^{\prime} w_4 + \delta^{\prime} w_3$;
	\item[\textbf{(2')}] There exist a minimal tripotent $w\in D_{\sigma(i)}$, a rank two tripotent $\tilde{u}\in D_{\sigma(i)}$, and complex numbers $\alpha^{\prime}, \beta^{\prime}, \delta^{\prime}$ such that $(\Delta(e), \tilde{u}, w)$ is a trangle, $|\alpha^{\prime}|^2 +2 | \beta^{\prime}|^2 + |\delta^{\prime}|^2 =1$, $\alpha^{\prime} \delta^{\prime}  = (\beta^{\prime})^2$, and $\Delta(v) = \alpha^{\prime} \Delta(e)+ \beta^{\prime} \tilde{u} +\delta^{\prime} w$.
\end{minipage}
\right.
\)\smallskip\medskip

\textbf{Case 1:} $\Delta( i e ) = i \Delta (e)$.\smallskip

Assume first that \textbf{(1)} and \textbf{(1')} hold. By the formula measuring the distance between minimal tripotents \eqref{eq formula for the gap metric between minimal tripotents} (cf. \cite[Proposition 3.3]{FerPe18Adv}), the hypothesis on $\Delta$ and \eqref{eq Delta preserves antipodal} we have $${(1 \pm \Re\hbox{e} (\alpha))+ \sqrt{(1 \pm  \Re\hbox{e} (\alpha))^2 -  |\delta|^2}}= \|e \pm v \|^2$$ $$= \|\Delta(e) \pm \Delta(v) \|^2= {(1 \pm \Re\hbox{e} (\alpha'))+ \sqrt{(1 \pm  \Re\hbox{e} (\alpha'))^2 -  |\delta'|^2}},$$ equivalently, \begin{equation}
	\label{eq system 1} (\mp \Re\hbox{e} (\alpha') \pm \Re\hbox{e} (\alpha)) + \sqrt{(1 \pm  \Re\hbox{e} (\alpha))^2 -  |\delta|^2}= \sqrt{(1 \pm   \Re\hbox{e} (\alpha'))^2 -  |\delta'|^2}.
\end{equation} By squaring both terms in the equations and subtracting the resulting identities we get $$ 2(\Re\hbox{e} (\alpha') - \Re\hbox{e} (\alpha))\left(2 + \sqrt{(1 -  \Re\hbox{e} (\alpha))^2 -  |\delta|^2} + \sqrt{(1 +  \Re\hbox{e} (\alpha))^2 -  |\delta|^2}  \right) =0,$$ which implies that $\Re\hbox{e} (\alpha) = \Re\hbox{e} (\alpha')$.\smallskip

Now, by applying that $\Delta( i e ) = i \Delta (e)$ and repeating the above arguments we have $$\begin{aligned} &{(1 \pm \Im\hbox{m} (\alpha))+ \sqrt{(1 \pm  \Im\hbox{m} (\alpha))^2 -  |\delta|^2}}= \| i e \pm v \|^2 = \|\Delta(i e) \pm \Delta(v) \|^2 \\ 
&= \|i \Delta( e) \pm \Delta(v) \|^2= {(1 \pm \Im\hbox{m} (\alpha'))+ \sqrt{(1 \pm  \Im\hbox{m} (\alpha'))^2 -  |\delta'|^2}},
\end{aligned}$$ leading to $\Im\hbox{m} (\alpha) =  \Im\hbox{m} (\alpha')$, and hence $\alpha = \alpha'$. We have therefore proved that $$TTP(v,e) = \alpha = \alpha' = TTP(\Delta(v),\Delta(e))$$ as desired.\smallskip

The arguments in cases \textbf{(1)} and \textbf{(2')}, \textbf{(2)} and \textbf{(1')}, and \textbf{(2)} and \textbf{(2')} are exactly the same, or even particular cases, and all lead to $TTP(v,e) = \alpha = \alpha' = TTP(\Delta(v),\Delta(e))$. This concludes the proof of the first statement in $({\color{red}14.a})$.\smallskip

\textbf{Case 2:} $\Delta( i e ) = - i \Delta (e)$.\smallskip

Assuming that \textbf{(1)} and \textbf{(1')} hold we have $${(1 \pm \Re\hbox{e} (\alpha))+ \sqrt{(1 \pm  \Re\hbox{e} (\alpha))^2 -  |\delta|^2}}= \|e \pm v \|^2$$ $$= \|\Delta(e) \pm \Delta(v) \|^2= {(1 \pm \Re\hbox{e} (\alpha'))+ \sqrt{(1 \pm  \Re\hbox{e} (\alpha'))^2 -  |\delta'|^2}},$$ and $$\begin{aligned} &{(1 \pm \Im\hbox{m} (\alpha))+ \sqrt{(1 \pm  \Im\hbox{m} (\alpha))^2 -  |\delta|^2}}= \| i e \pm v \|^2 = \|\Delta(i e) \pm \Delta(v) \|^2 \\ 
	&= \|- i \Delta( e) \pm \Delta(v) \|^2= {(1 \mp \Im\hbox{m} (\alpha'))+ \sqrt{(1 \mp  \Im\hbox{m} (\alpha'))^2 -  |\delta'|^2}},
\end{aligned},$$ equations which combined give $TTP(v,e) = \alpha = \overline{\alpha'} = \overline{TTP(\Delta(v),\Delta(e))}$.  The other possible cases can be similarly treated, and all together prove the first statement in $({\color{red}14.b})$.\smallskip

Let us now prove the final claims in $({\color{red}14.a})$ and  $({\color{red}14.b})$. Suppose on the contrary that we can find $e,v\in \mathcal{U}_{min}(C_{i})$ such that $\Delta(i e) = i \Delta(e)$ and $\Delta(i v) = - i \Delta(v)$.\smallskip

We shall first show that $\Delta(i w) = i \Delta(w)$ (respectively, $\Delta(i w) = - i \Delta(w)$) for all $w\in \mathcal{U}_{min}(C_{i})$ with $TTP(w,e) \neq 0$ (respectively, $TTP(w,v) \neq 0$). Namely, for each $w\in \mathcal{U}_{min}(C_{i})$ we know that $\Delta (i w )\in \{\pm i \Delta(w)\}$ (cf. \eqref{eq Delta i times minimal}). If $\Delta (i w ) =-i \Delta (w )$ (respectively, $\Delta (i w ) =i \Delta (w )$) it follows from the first part of $({\color{red}14.a})$ (respectively, from the first part of $({\color{red}14.b})$) that 
$$\begin{aligned}
-i TTP(w , e) &= TTP(-i \Delta (w ), \Delta (e)) = TTP(\Delta (i w ), \Delta (e))\\
& = TTP(i w , e) = i TTP(w, e),
\end{aligned}$$ (respectively,  
$$\begin{aligned}
i \ \overline{TTP(w , v)}& = TTP(i \Delta (w ), \Delta (v)) = TTP(\Delta (i w ), \Delta (v)) \\ 
&= \overline{TTP(i w , e)} = - i \overline{TTP(w, e)} \Big)
\end{aligned}$$ which forces to the condition $TTP(w,e) =0$ (respectively, $TTP(w,v) =0$).\smallskip

We deduce from the above paragraphs and the assumptions on $e$ and $v$ that $TTP(v,e) =0 = \overline{TTP(e,v)}.$ Combining this information with the result describing the relative position of two minimal tripotents in \cite[Lemma 3.10]{FerPe18Adv} (as employed in many cases before), we can assume the existence of minimal tripotents $v_2,v_3,v_4$ in $C_{i}$ and complex numbers $\beta$ and $\delta$ such that $(e,v_2,v_3,v_4)$ is a quadrangle, $| \beta|^2 + |\delta|^2 =1$, and $v =  \beta v_2 + \delta v_3$. The element $u = \frac{1}{\sqrt{2}} e +  \frac{1}{\sqrt{2}} v_2$ is a minimal tripotent in $C_{i}$ with $TTP(u,e) =  \frac{1}{\sqrt{2}} \neq 0$ and $TTP (u, v) = \frac{\overline{\beta}}{\sqrt{2}}$. The previous conclusion shows that $\beta =0,$ and hence $v =  \delta v_3$, for a unitary $\delta\in \mathbb{T}$. In such a case $ \tilde{u}= \frac12 (e + v_2 + v_3+ v_4)$ is a minimal tripotent such that $TTP(\tilde{u},e) = \frac12\neq 0$ and $TTP(\tilde{u}, v) = \frac{\overline{\delta}}{2}\neq 0,$ which is impossible. This concludes the proof of $({\color{red}14.a})$ and $({\color{red}14.b})$. \smallskip 

Let us define $\Lambda_{1,0} := \Big\{ i \in \Lambda_1 : C_i \hbox{ has rank-one}\Big\},$
$$\Lambda_{1,l} := \Big\{ i \in \Lambda_1 : C_i \hbox{ has rank } \geq 2 \hbox{ and } \exists e\in \mathcal{U}_{min} (C_i) \hbox{ with } \Delta (i e) = i \Delta(e)\Big\},$$
and $$\Lambda_{1,c} := \Big\{ i \in \Lambda :  C_i \hbox{ has rank } \geq 2 \hbox{ and } \exists e\in \mathcal{U}_{min} (C_i) \hbox{ with } \Delta (i e) =- i \Delta(e)\Big\}.$$ We deduce from $({\color{red}14.a})$ and $({\color{red}14.b})$ that $\Delta(i e) = i \Delta(e)$ for every $i \in \Lambda_{1,l},$ and every $e\in \mathcal{U}_{min} (C_i)$ and $\Delta(ie) = - i \Delta(e)$ for every $i \in \Lambda_{1,c},$ and each $e\in \mathcal{U}_{min} (C_i).$\smallskip

For each $i \in \Lambda_{1,0}$ there exists a conjugation $j_i$ on  $D_{\sigma(i)},$ a real linear surjective isometry $T_i: C_{i}\to D_{\sigma(i)}$ such that $T_i (e) = \Delta (e)$ for all $e\in \mathcal{U}_{min}(C_{i}) = S_{C_{i}}$ and $j_i \circ T_i : C_{i}\to D_{\sigma(i)}$ is an isometric (complex linear) triple isomorphism (cf. Remark \ref{r surjective real linear isometries}). For $i\in \Lambda_{1,c}$ we can find a conjugation $j_i$ on $D_{\sigma(i)}$ (the existence is guaranteed by \cite[Theorem 4.1]{Ka97}). For $i\in \Lambda_{1,l}$ we set $j_i = Id_{D_{\sigma(i)}}$. Define a real linear mapping $J: N = \bigoplus_{j\in \Lambda_2}^{\ell_{\infty}} {D}_j \to N = \bigoplus_{j\in \Lambda_2}^{\ell_{\infty}} {D}_j,$ by $J((y_{\sigma(i)})_{i\in \Lambda_1}) := (j_i (y_{\sigma(i)}))_{i\in \Lambda_1}.$ The mapping $J$ is a real linear surjective isometry, and by construction, for $i \in \Lambda_{1,0}\cup \Lambda_{1,c}$ and $e,v\in \mathcal{U}_{min} (C_i)$ we have $J \Delta(i e) = j_i \Delta(i e) = i j_i \Delta( e) = i J\Delta(e),$ and thus $$TTP(J \Delta(e), J\Delta(v)) = TTP(j_i \Delta(e), j_i \Delta(v))= TTP(e, v).$$ Clearly, $TTP(J \Delta(e), J\Delta(v)) = TTP(e, v),$ for all $i \in \Lambda_{1,l}$ and $e,v\in \mathcal{U}_{min} (C_i)$, and hence the same conclusion holds for all $e,v\in \mathcal{U}_{min} (M)$ by orthogonality. We have therefore shown that the mapping $J\Delta: \mathcal{U}_{min}(M)\to \mathcal{U}_{min}(N)$ is a surjective isometry preserving triple transition pseudo-probabilities. Corollary \ref{c bijections preserving triple transition pseudo-probabilities and orthogonality} asserts that $J\Delta$ extends {\rm(}uniquely{\rm)} to a surjective complex-linear {\rm(}isometric{\rm)} triple isomorphism $\Phi$ from $M$ onto $N$. Finally, the mapping $J^{-1} \Phi : M\to N$ is a surjective real linear isometry whose restriction to $\mathcal{U}_{min} (M)$ is $\Delta$.
\end{proof}

\smallskip\smallskip

\textbf{Acknowledgements} Author partially supported by Project PID2021-122126NB-C31, financed by: ERDF / Ministry of Science and Innovation  - State Research Agency, Junta de Andaluc\'{\i}a grants FQM375 and PY20$\underline{\ }$00255, and by the IMAG--Mar{\'i}a de Maeztu grant CEX2020-001105-M/AEI/10.13039/ 501100011033.

\end{document}